\documentclass[a4paper,12pt, leqno]{amsart}
\usepackage{amsmath,amsthm,amsfonts,amssymb,graphicx,mathrsfs,esint,longtable,mathtools}

\newcommand{\R}{{\mathbb R}}

\newcommand{\haus}{{\mathcal H}}

\newcommand{\Ss}{\mathbb{S}}

\newcommand{\N}{\mathbb{N}}

\newcommand{\ph}{\varphi}
\newcommand{\e}{\varepsilon}
\newcommand{\la}{\lambda}
\newcommand{\diam}{\textnormal{diam$\,$}}

\newcommand{\Ha}{\mathcal{H}}

\newcommand{\zero}{0}

\newtheorem{theorem}{\textbf{THEOREM}}[section]
\newtheorem{lemma}[theorem]{\textsc{Lemma}}
\newtheorem{proposition}[theorem]{\textsc{Proposition}}
\newtheorem{corollary}[theorem]{\textsc{Corollary}}
\theoremstyle{definition}
\newtheorem{definition}[theorem]{\textsc{Definition}}
\newtheorem{question}[theorem]{\textsc{Question}}
{\theoremstyle{remark} \newtheorem{remark}[theorem]{Remark}}
\def\charfn_#1{{\raise1.2pt\hbox{$\chi_{\kern-1pt\lower3pt\hbox{{$\scriptstyle#1$}}}$}}}
\def\leq{\leqslant }
\def\geq{\geqslant }

\def\XXint#1#2#3{{\setbox0=\hbox{$#1{#2#3}{\int}$}
\vcenter{\hbox{$#2#3$}}\kern-.5\wd0}}

\def\le {\leqslant}
\def\ge {\geqslant}
\begin{document}
\title{Quasispheres and metric doubling measures}
\author{Atte Lohvansuu, Kai Rajala and Martti Rasimus} \thanks{Research supported by the Academy of Finland, project number 308659. \newline {\it 2010 Mathematics Subject Classification.} Primary 30L10, Secondary 30C65, 28A75. 
}
\date{}

\begin{abstract} 
Applying the Bonk-Kleiner characterization of Ahlfors $2$-regular quasispheres, we show that a metric two-sphere $X$ 
is a quasisphere if and only if $X$ is linearly locally connected and carries a \emph{weak metric doubling measure}, i.e., a measure that deforms the metric on $X$ without much shrinking. 
\end{abstract}
\maketitle

\renewcommand{\baselinestretch}{1.2}


\section{Introduction}
A homeomorphism $f \colon (X,d) \to (Y,d')$ between metric spaces is \emph{quasisymmetric}, if there exists a homeomorphism $\eta \colon [0,\infty) \to [0, \infty)$ such that 
$$
\frac{d(x_1,x_2)}{d(x_1,x_3)} \le t \, \text{ implies } \, \frac{d'(f(x_1),f(x_2))}{d'(f(x_1),f(x_3))} \le \eta(t) 
$$
for all distinct $x_1,x_2,x_3 \in X$. Applying the definition with $t=1$ shows that quasisymmetric homeomorphisms map all balls to sets that are uniformly round. Therefore, the condition of quasisymmetry can be seen as a global version of conformality or quasiconformality. 

Starting from the work of Tukia and V\"ais\"al\"a \cite{TV}, a rich theory of quasisymmetric maps between metric spaces has been developed. An overarching problem is to characterize the metric spaces that can be mapped to a given space $S$ by a quasisymmetric map. 

This problem is particularly appealing when $S$ is the two-sphere $\Ss^2$. There are connections to geometric group theory, (cf. \cite{Bon}, \cite{BK2}, \cite{BonMer}), complex dynamics (\cite{BLM}, \cite{BM}, \cite{HP}), as well as minimal surfaces (\cite{LW}). 

Bonk and Kleiner \cite{BK} solved the problem in the setting of two-spheres with ``controlled geometry", see also \cite{LW}, \cite{MeWi}, \cite{Raj}, \cite{Se1}, \cite{Wil}. We say that $(X,d)$ is a \emph{quasisphere}, if there is a quasisymmetric map from $(X,d)$ to $\Ss^2$. See Section \ref{sec:prel} for further definitions. 

\begin{theorem}[\cite{BK}, Theorem 1.1] 
\label{BKtheorem}
Suppose $(X,d)$ is homeomorphic to $\Ss^2$ and Ahlfors $2$-regular. Then $(X,d)$ is a quasisphere if and only if it is linearly locally connected. 
\end{theorem}

Finding generalizations of the Bonk-Kleiner theorem beyond the Ahlfors $2$-regular case and to fractal surfaces is important; applications include Cannon's conjecture on hyperbolic groups, cf. \cite{Bo}, \cite{Kl} (by \cite{Co} the boundary of a hyperbolic group is Ahlfors $Q$-regular with $Q$ greater than or equal to the topological dimension of the boundary). A characterization of general quasispheres in terms of combinatorial modulus is given in \cite[Theorem 11.1]{BK}. However, this result is difficult to apply in practice and in fact an easily applicable characterization is not likely to exist. Several types of fractal quasispheres have been found (cf. \cite{Bi}, \cite{DT}, \cite{Me1}, \cite{V}, \cite{VW}, \cite{Wu}), showing the difficulty of the problem. 

In this paper we characterize quasispheres in terms of a condition related to \emph{metric doubling measures} 
of David and Semmes \cite{DS1}, \cite{DS2}. These are measures that deform a given metric in a controlled manner. 
More precisely, a (doubling) Borel measure $\mu$ is a metric doubling measure of dimension $2$ on $(X,d)$ if there is a metric $q$ on $X$ and $C \geq 1$ 
such that for all $x,y \in X$, 
\begin{equation}
\label{kantapaa} 
C^{-1} \mu(B(x,d(x,y)))^{1/2} \le q(x,y) \le  C \mu(B(x,d(x,y)))^{1/2}. 
\end{equation}
It is well-known that metric doubling measures induce quasisymmetric maps $(X,d) \to (X,q)$. Our main result shows that quasispheres can be characterized using a weaker condition where we basically only assume the first inequality of \eqref{kantapaa}. We call measures satisfying such a condition \emph{weak metric doubling measures}, see Section \ref{sec:prel}. 

\begin{theorem} \label{mainthm}
Let $(X,d)$ be a metric space homeomorphic to $\Ss^2$. Then $(X,d)$ is a quasisphere if and only if it is linearly locally connected and carries a weak metric doubling measure of dimension $2$. 
\end{theorem}

To prove Theorem \ref{mainthm} we show, roughly speaking, that the first inequality in \eqref{kantapaa} actually implies the second 
inequality. It follows that $\mu$ induces a quasisymmetric map $(X,d) \to (X,q)$, and $(X,q)$ is $2$-regular and linearly locally connected.  Applying Theorem \ref{BKtheorem} to $(X,q)$ and composing then gives the desired quasisymmetric map. It would be interesting to find higher-dimensional as well as quasiconformal versions of Theorem \ref{mainthm}. See Section \ref{sec:rem} for further discussion. 


\section{Preliminaries} 
\label{sec:prel}
We first give precise definitions. Let $X=(X,d)$ be a metric space. As usual, $B(x,r)$ is the open ball in $X$ with center $x$ and radius $r$, and $S(x,r)$ is the set of points whose distance to $x$ equals $r$. 

\vskip5pt
We say that $X$ is \emph{$\lambda$-linearly locally connected} (LLC), if for any $x \in X$ and $r>0$ it is possible to join any two points in $B(x,r)$ with a continuum in $B(x,\la r)$, and any two points in $X \setminus B(x,r)$ with a continuum in $X \setminus B(x,r / \la)$.

\vskip5pt

A Radon measure $\mu$ on $X$ is \emph{doubling}, if there exists a constant $C_D \ge 1$ such that for all $x \in X$ and $R > 0$, 
\begin{equation}
\label{doubl}
\mu(B(x,2R)) \le C_D \mu(B(x,R)), 
\end{equation}
and \emph{Ahlfors $s$-regular}, $s>0$, if there exists a constant $A \ge 1$ such that for all $x \in X$ and $0<R< \diam X$, 
$$
A^{-1}R^s\le \mu(B(x,R)) \le AR^s. 
$$
Moreover, $X$ is Ahlfors $s$-regular if it carries an $s$-regular measure $\mu$. 
\vskip5pt

We now define weak metric doubling measures. In what follows, we use notation $B_{xy} = B(x,d(x,y)) \cup B(y,d(x,y))$. 
\vskip5pt

Let $\mu$ be a doubling measure on $(X,d)$. For $x,y \in X$ and $\delta>0$, a finite sequence of points $x_0,x_1,\dots, x_m$ in $X$ 
is a \emph{$\delta$-chain from $x$ to $y$}, if $x_0=x$, $x_m=y$ and $d(x_j,x_{j-1}) \le \delta$ for every $j=1,\dots, m$. 

\vskip5pt

Now fix $s>0$ and define a ``$\mu$-length" $q_{\mu,s}$ as follows: set 
$$
q_{\mu,s}^\delta(x,y) \coloneqq \inf \Big\{ \sum_{j=1}^m \mu(B_{x_j x_{j-1}})^{1/s} \colon (x_j)_{j=0}^m \text{ is a $\delta$-chain from $x$ to $y$} \Big\}
$$
and 
$$
q_{\mu,s}(x,y) \coloneqq \limsup_{\delta \to \zero} q_{\mu,s}^\delta(x,y) \in [0,\infty]. 
$$

\begin{definition} \label{wmdm}
A doubling measure $\mu$ on $(X,d)$ is a \emph{weak metric doubling measure of dimension $s$}, if there exists $C_W \ge 1$ such that for all $x, y \in X$, 
\begin{equation}
\label{weakd}
\frac{1}{C_W}  \mu(B_{xy})^{1/s} \le q_{\mu,s}(x,y). 
\end{equation}
In what follows, if the dimension $s$ is not specified then it is understood that $s=2$, and $q_{\mu,2}$ is shortened to $q_{\mu}$. 

\end{definition}




\section{Proof of Theorem \ref{mainthm}}
In this section we give the proof of Theorem \ref{mainthm}, assuming Proposition \ref{estimate} to be proved in the following sections. 
First, it is not difficult to see that if there exists a quasisymmetric map $\ph :X \to \Ss^2$, then $X$ is LLC, and 
$$ 
\mu(E) \coloneqq \haus^2(\ph(E))
$$
defines a weak metric doubling measure on $X$. Therefore, the actual content in the proof of Theorem \ref{mainthm} is the existence of a quasisymmetric parametrization, assuming LLC and the existence of a weak metric doubling measure (of dimension $2$). The proof is based on the following result. 

\begin{proposition} \label{estimate}
Let $(X,d)$ be LLC and homeomorphic to $\Ss^2$. Moreover, assume that $(X,d)$ carries a weak metric doubling measure $\mu$ 
of dimension $2$. Then $q_{\mu}$ is a metric on $X$ and $\mu$ is a metric doubling measure in $(X,q_\mu)$, that is there exists a constant $C_S \ge 1$ such that also the bound
$$ 
q_{\mu}(x,y) \le C_S\mu(B_{xy})^{1/2}
$$
holds for all $x,y \in X$.
\end{proposition}

We will apply the well-known growth estimates for doubling measures. The proof is left as an exercise, see \cite[ex. 13.1]{heinonen}. 

\begin{lemma}\label{lemma:tuplauslemma}
Let $X$ be as in Proposition \ref{estimate} and let $\mu$ be a doubling measure on $X$. Then there exist constants $C, \alpha>1$ depending only on the doubling constant $C_D$ of $\mu$ such that
\[
\frac{\mu(B(x, r_2))}{\mu(B(x, r_1))}\leq C\max\left\lbrace\left(\frac{r_2}{r_1}\right)^\alpha, \left(\frac{r_2}{r_1}\right)^{1/\alpha} \right\rbrace
\]
for all $0<r_1, r_2<\mathrm{diam}(X)$.
\end{lemma}

Combining Proposition \ref{estimate} and Lemma \ref{lemma:tuplauslemma} shows that $q_{\mu}$ induces a quasisymmetric map. This is essentially Proposition 14.14 of \cite{heinonen}. We include a proof for completeness. 

\begin{corollary} \label{regular}
Let $X$ and $\mu$ be as in Proposition \ref{estimate}. Then the identity mapping $i \colon (X,d) \to(X,q_{\mu})$ is quasisymmetric, and 
$(X,q_{\mu})$ is Ahlfors $2$-regular.
\end{corollary}

\begin{proof}
We denote $q=q_{\mu}$. We first show that $i$ is a homeomorphism. Since $(X, d)$ is a compact metric space, it suffices to show that $i$ is continuous, i.e., that any $q$-ball $B^q(x, r)$ contains a $d$-ball $B^d(x, \delta)$ for some $\delta=\delta(x, r)$. Suppose that this does not hold for some $x\in X$ and $r>0$. Then there exists a sequence $(x_n)_{n=1}^\infty$ such that $d(x_n, x)\rightarrow 0$ but $q(x_n, x)\geq r$ for all $n\in\N$. Now Proposition \ref{estimate} implies 
\[
r\leq q(x_n, x)\leq C\mu(B^d(x, 2d(x, x_n)))^{1/2}\overset{n\rightarrow\infty}{\longrightarrow}0,
\]
which is a contradiction. Thus $i$ is a homeomorphism. Let $x, y, z\in X$ be distinct. By Proposition \ref{estimate} and Lemma \ref{lemma:tuplauslemma} we have 
\[
\frac{q(x, y)}{q(x, z)}\leq C\frac{\mu(B_{xy})^{1/2}}{\mu(B_{xz})^{1/2}}\leq C\frac{\mu(B(x, 2d(x, y)))^{1/2}}{\mu(B(x, 2d(x, z)))^{1/2}}\leq \eta\left(\frac{d(x, y)}{d(x, z)}\right),
\]
where $\eta: [0, \infty)\rightarrow [0, \infty)$ is the homeomorphism
\[
\eta(t)=C\max\{t^{\alpha/2}, t^{1/2\alpha}\}.
\]
Thus $i$ is $\eta$-quasisymmetric. 

We next claim that $\mu$ is Ahlfors $2$-regular on $(X,q)$. Fix $x \in X$ and $0<r< \diam(X,q)/10$. Since $(X,q)$ is connected, 
there exists $y \in S^q(x,r)$. Now by Proposition \ref{estimate}, 
$$
C_S^{-2} r^2 \leq \mu(B_{xy}) \leq C_W^2 r^2. 
$$
On the other hand, the quasisymmetry of the identity map $i$ and the doubling property of $\mu$ give 
$$
C^{-1} \mu(B^q(x,r))  \leq \mu(B_{xy}) \leq C \mu(B^q(x,r)),  
$$
where $C$ depends only on $C_D$ and $\eta$. Combining the estimates gives the $2$-regularity. 
\end{proof}

We are now ready to finish the proof of Theorem \ref{mainthm}, modulo Proposition \ref{estimate}. Indeed, Corollary \ref{regular} shows that there is a quasisymmetric map from $(X,d)$ onto the $2$-regular $(X,q_{\mu})$. It is not difficult to see that the quasisymmetric image of a LLC space is also LLC. Hence, by Theorem \ref{BKtheorem}, there exists a quasisymmetric map from $(X,q_{\mu})$ onto $\Ss^2$. Since the composition of two quasisymmetric maps is quasisymmetric, Theorem \ref{mainthm} follows. 

\section{Separating chains in annuli}
We prove Proposition \ref{estimate} in two parts. In this section we find short chains in annuli (Lemma \ref{lemma:annuluskontinuumi}). 
In the next section we take suitable unions of these chains to connect given points. 

We first show that it suffices to consider $\delta$-chains with sufficiently small $\delta$. In what follows, we use notation 
$$
cB_{xy} = B(x,cd(x,y)) \cup B(y,cd(x,y)).
$$

\begin{lemma} \label{tarkkuus}

Let $(X,d)$ be a compact, connected metric space admitting a weak metric doubling measure $\mu$ of some dimension $s>0$. Then for any $r>0$ there exists $\delta_r>0$ such that if $x,y \in X$ with $d(x,y) \ge r$ then we have
\begin{equation}
\label{grope}
2C_WC_D^{2/s} q_{\mu,s}^{\delta_r}(x,y) \ge \mu (B_{xy})^{1/s},
\end{equation}
where $C_W$ and $C_D$ are the constants in \eqref{weakd} and \eqref{doubl}, respectively. 
\end{lemma}

\begin{proof}
Suppose to the contrary that \eqref{grope} does not hold for some $r>0$. Then there exists a sequence of pairs of points $(x_j,y_j)_j$ for which $d(x_j,y_j) \ge r$ and
$$ q_{\mu,s}^{1/j}(x_j,y_j) < \frac{1}{2C_WC_D^{2/s}} \mu(B_{x_j y_j})^{1/s}$$
for all $j=1,2,3,\dots.$ Then by compactness we can, after passing to a subsequence, assume that $x_j \to x$ and $y_j \to y$ where also $d(x,y) \ge r$. Let then $k \in \N$ be arbitrary and $j \ge k$ so large that $B_{x_j y_j} \subset 4 B_{xy}$,
$$d(x,x_j),d(y,y_j) \le \frac1k$$
and
\begin{equation}
\label{koo1}
\mu(B_{x x_j})^{1/s} + \mu(B_{y y_j})^{1/s} < \frac1{3C_W}\mu(B_{xy})^{1/s} .
\end{equation}
The last estimate is made possible by the fact that $\mu(\{z\})=0$ for every point $z$ in the case of a doubling measure and a connected space, or more generally when the space is \emph{uniformly perfect} (see \cite[5.3 and 16.2]{DS2}). Now choose a $\frac1j$-chain $z_0,\dots,z_m$ from $x_j$ to $y_j$ satisfying 
\begin{equation}
\label{koo2}
 \sum_{i=1}^m \mu(B_{z_i z_{i-1}})^{1/s} < \frac{1}{2C_WC_D^{2/s}} \mu(B_{x_j y_j})^{1/s} \leq \frac{1}{2C_W} \mu(B_{x y})^{1/s}
\end{equation}
so that $x,z_0,\dots,z_m,y$ is in particular a $\frac1k$-chain from $x$ to $y$. Combining \eqref{koo1} and \eqref{koo2}, we have 
$$ q_{\mu,s}^{1/k}(x,y) < \frac5{6C_W}\mu(B_{xy})^{1/s}. $$
This contradicts \eqref{weakd} when $k \to \infty$. 
\end{proof}

In what follows, we will abuse terminology by using a non-standard definition for separating sets. 

\begin{definition}
Given $A,B,K \subset X$, we say that $K$ \emph{separates} $A$ and $B$ if there are distinct connected components $U$ and $V$ of $X \setminus K$ such that $A \subset U$ and $B \subset V$. 
\end{definition}

\begin{lemma}\label{lemma:annuluskontinuumi}
Suppose $(X,d)$ is $\la$-LLC and homeomorphic to $\Ss^2$, and $\mu$ a weak metric doubling measure on $X$. Let $k$ be the smallest integer such that $2^k > \la$. Then there exists $C >1$ depending only on $\lambda$, $C_D$ and $C_W$ such that for any $x \in X, 0<r<2^{-8k}\diam X$ and $\delta >0$ there exists a $\delta$-chain $x_0,\dots,x_p$ in the annulus $\overline{B(x,2^{5k}r)} \setminus B(x,2^{2k}r)$ such that
$$\sum_{j=1}^p \mu(B_{x_jx_{j-1}})^{1/2} \le C \mu(B(x,r))^{1/2}$$
and the union $\cup_j \overline{5B_{x_j x_{j-1}}}$ contains a continuum separating $B(x,r)$ and $X \setminus \overline{B(x,2^{7k}r)}$.
\end{lemma}

\begin{proof}

Let $x \in X, 0<r<2^{-8k}\diam X$ and $\delta >0$ be arbitrary. By Lemma \ref{tarkkuus} we may assume without loss of generality that
\begin{equation} q_{\mu}^{\delta}(y,z) \ge \frac1{C'} \mu(B_{yz})^{1/2}  \label{alaraja} \end{equation} 
for any $y \in S(x,2^{3k} r), z \in S(x,2^{4k}r)$ and also $\delta < r$ by finding a finer chain than possibly asked.

Next we cover the annulus $A=\overline{B(x,2^{5k}r)} \setminus {B(x,2^{2k} r)}$ as follows: Let $\e>0$ be small enough so that $\mu(B(w,\delta/10)) > \e^2$ for every $w \in X$ (see again \cite[16.2]{DS2}). Then for every $w \in A$ we can choose a radius $0<r_w<\delta/10$ with
$$ \frac{\e^2}{2C_D} \le \mu(B(w,r_w)) \le \e^2.  
$$
Using the $5r$-covering theorem, we find a finite number $m$ of pairwise disjoint balls $B_j=B(w_j, r_j)$, $r_j=r_{w_j}$ from the cover $\{B(w,r_w) \}_{w \in A}$, such that
$$ A \subset \bigcup_{j=1}^m 5B_j \subset B(x,2^{6k}r) \setminus \overline{B(x,2^kr)}.$$
Observe that for any point $z $ in the thinner annulus $A'=\overline{B(x,2^{4k}r)} \setminus {B(x,2^{3k} r)}$ there exists a continuum in $A$ joining $z$ to some point $y \in S(x,2^{3k}r)$ by the LLC-property. Hence there exists a subcollection $B_1',\dots,B_n'$ of the cover $(5B_j)$ forming a ball chain from this $y$ to $z$, meaning that $y \in B_1', z \in B_n'$ and $B_j' \cap B_{j+1}' \ne \emptyset$. Thus we can define a ``counting'' function $u$ for this cover on $A'$ by setting $u(z)$ to be the smallest $n \in \{1,\dots,m\}$ so that there exists a ball chain $(B_i')_{i=1}^n$ from some $y \in  S(x,2^{3k} r)$ to $z.$

Using \eqref{alaraja}, we find a lower bound for $u$ on $ S(x,2^{4k} r)$: Let $y \in  S(x,2^{3k} r), z \in  S(x,2^{4k} r)$ be arbitrary and $(B_i')_{i=1}^n=(B(w_i',5r_i'))_{i=1}^n$ the corresponding chain. Then $y=w_0', w_1', \dots, w_n', z=w_{n+1}'$ is also a $\delta$-chain. Hence
$$ \mu(B_{yz})^{1/2} \le C' \sum_{i=1}^{n+1} \mu(B_{w'_iw'_{i-1}})^{1/2} \le C' C_D^3 n \e$$
as every $B_{w'_i w'_{i-1}}$ is contained in $B(w'_l,20r_{w_l'})$, $l=i$ or $i-1$. On the other hand $B(x,2^{7k}r) \subset B(y,2^{7k+1}r)$, and since the balls $B_j$ are disjoint, 
$$ m \e^2 \le \mu(B(y,2^{7k+1}r)) \le C_D^{7k+1} \mu(B_{yz}),$$
implying $n^2 \ge m/C''$ or $u(z) \ge \sqrt{m}/C''$. 

Let then $n$ be the minimal value of $u$ on $S(x,2^{4k}r)$ and for $j=1,2,\dots,n$ define
$$ A_j = \bigcup_{5B_i \cap u^{-1}(j) \ne \emptyset} 5B_i.$$
By the definition of $u$ each ball $5B_i$ can be contained in at most two ``level sets'' $A_j$ and so we obtain a constant $C \ge 1$ such that
\begin{align*}
\min_{1 \le j \le n} \sum_{5B_i \subset A_j} \mu(5B_i)^{1/2} & \le \frac1n \sum_{j=1}^n \sum_{5B_i \subset A_j} \mu(5B_i)^{1/2} \\
& \le \frac1n C_D^3 \e\cdot  2m \\
& \le 2 C_D^3 \frac{\sqrt{m}}{n}\sqrt{\e^2 m} \\
& \le C \mu(B(x,r))^{1/2}.
\end{align*}
Let $j \in \{1,\dots, n\}$ be the index for which the above left hand sum is smallest. Since by construction $A_j$ necessarily intersects any curve joining $B(x,2^kr)$ and $X \setminus \overline{B(x,2^{6k}r)}$, it separates $B(x,r)$ and $X \setminus \overline{B(x,2^{7k}r)}$ by the LLC-property as $2^k > \la$. Hence the closed set $\overline{A_j}$ contains a continuum $K$ separating these sets by topology of $\Ss^2$, see for example \cite{Ne} V 14.3.. Now $K$ is covered by a ball chain $\overline{B(w_0',5r'_{0})}, \dots, \overline{B(w_p',5r'_{p})}$ of closures of balls $5B_i$ contained in $A_j$. Hence these points $w_0',\dots,w_p'$ are the desired $\delta$-chain, since clearly $ d(w_i',w_{i+1}') \le 5r_i'+5r_{i+1}'< \delta$ and
$$ \sum_{i=1}^p \mu(B_{w'_iw'_{i-1}})^{1/2} \le C \mu(B(x,r))^{1/2}$$
by our choice of $j$.
\end{proof}

\begin{remark}
Note that in the claim of the above lemma the constant $C$ is uniform with respect to the required step $\delta$ of the chain; we can in fact find arbitrarily fine chains and have the same estimate from above for $\sum \mu(B_j)^{1/2}$. This is essentially obtained by the doubling property and the $5r$-covering theorem. We also work with dimension $s=2$, since passing from the lower estimate of \ref{tarkkuus} to the upper in the claim we actually switch the power $1/s$ of the measure to $(s-1)/s$, both $1/2$ in the proof. Thus this argument seems not to apply for higher dimension (see Question \ref{high}). Moreover the topology of $\Ss^2$ is used for finding a single separating component, which is not always possible for example on a torus.
\end{remark}



\section{Proof of Proposition \ref{estimate}}
In this section $(X,d,\mu)$ satisfies the assumptions of Proposition \ref{estimate}. Lemma \ref{lemma:annuluskontinuumi} and the $5r$-covering lemma then give the following: For any given $B=B(x, R) \subset X$ and $\delta>0$ there is a cover of the $x$-component $U$ of $B$ by at most $M=M(\la,C_D, L)$ balls $\{B_i\}_{i=1}^m$ with centers in $U$ such that for every $i$
\begin{enumerate}
\item  $L^{-2}\mu(B)\leq\mu(B_i)\leq L^{-1}\mu(B)$
\item A continuum $K_i \subset \overline{2^{7k}B_i}\setminus B_i$ separates $B_i$ and $X\setminus \overline{2^{7k}B_i}$
\item $K_i \subset \bigcup_p\overline{5B_{x^{i}_px^{i}_{p-1}}}$, where $(x^{i}_p)_p$ is a $\delta$-chain
\item $\sum_p\mu(B_{x^{i}_px^{i}_{p-1}})^{1/2}\leq C\mu(B_i)^{1/2}$.
\end{enumerate}
Here $k$ is as in Lemma \ref{lemma:annuluskontinuumi}, $L>C_D^{8k}$ and $C=C(\lambda,C_D,C_W)$. 
 
We would like to take unions of the continua $K_i$ to join points. However, the union $\cup_i K_i$ need not be a connected set. The following lemma takes care of this problem. We denote by $\hat K_i$ the \emph{interior} of $K_i$, i.e., the component of $X\setminus K_i$ that contains $B_i$.

\begin{lemma}\label{lemma:kontinuumilemma}
Let $i\in \{1, 2\}$. Let $B_i=B(x_i, r_i)\subset X$ be a (small) ball and let $K_i\subset\overline{2^{7k}B_i}\setminus B_i$ be a continuum that separates $B_i$ and $X\setminus \overline{2^{7k}B_i}$. Suppose $\hat K_1\cap\hat K_2\neq\emptyset$. 
If $K_1\cap K_2=\emptyset$, then either $K_1\subset \hat K_2$ and $\hat K_1\subset \hat K_2$ or $K_2\subset\hat K_1$ and $\hat K_2\subset \hat K_1$. 
\end{lemma}

\begin{proof}
Since $X$ is homeomorphic to $\Ss^2$, path components of an open set in $X$ are exactly its components. In addition such components are open. Since $K_1$ and $K_2$ are nonempty disjoint compact sets, there exist path connected open sets $U_1, U_2\subset X$ such that $K_1\subset U_1\subset X\setminus K_2$ and $K_2\subset U_2\subset X\setminus K_1$. Let $w\in \hat K_1\cap \hat K_2$. Let $\gamma: [0, 1]\rightarrow X$ be a path from $w$ to $z\in X\setminus (\overline{2^{7k}B_1} \cup  \overline{2^{7k}B_2})$. By the separation properties $\gamma([0, 1])$ intersects $K_1$ and $K_2$. Let 
\[
s=\inf\{t\in [0, 1]\ |\ \gamma(t)\in K_1\cup K_2\}.
\]
Now $s>0$ and $\tilde\gamma:=\gamma|_{[0, s]}$ is a path that intersects $K_1\cup K_2$ exactly once. Without loss of generality we may assume $\gamma(s)\in K_1$. By construction of $U_1$ the point $w$ can be connected to any point in $K_1$ inside $X\setminus K_2$. Thus $K_1\subset\hat K_2$. Now  let $y\in \hat K_1$. It suffices to show that there exists a path in $\hat K_2$ from $y$ to $w$. Suppose there is no such path. Now the argument of the first part of this proof implies that $K_2\subset \hat K_1$. Let $S$ be the number obtained by changing the infimum in the definition of $s$ to the respective supremum. Necessarily $\gamma(S)\in K_2$, since otherwise we could construct a path in $\hat K_2$ from $w$ to $z$. Since $K_2\subset U_2\subset \hat K_1$, there exists a path connecting $w$ to $\gamma(S)$ in $\hat K_1$, i.e., there exists a path from $w$ to $z$ in $\hat K_1$, which is impossible. Thus $\hat K_1\subset \hat K_2$.
\end{proof}

Motivated by Lemma \ref{lemma:kontinuumilemma} we say that a continuum $K_i$ is \emph{maximal} (in $\{K_i\}_{i=1}^m$) if it is not contained in the interior of some other $K_j$. Define $K$ to be the union of all maximal continua in $\{K_i\}_{i=1}^m$. Clearly $K$ is compact. Let us show that it is also connected. Suppose $K_i$ and $K_j$ are distinct maximal continua.  Let $B_{(i)}$ and $B_{(j)}$ be the balls in $\{B_i\}$ that are contained in the interiors $\hat K_i$ and $\hat K_j$, respectively. Since $\{B_i\}$ is a cover of the $x$-component of $B$, we can find a chain of balls in $\{B_i\}$ connecting any pair of points in the component. On the other hand, every ball $B_i$ intersects the $x$-component, so it suffices to consider the case where $B_{(i)}\cap B_{(j)}\neq\emptyset$. By Lemma \ref{lemma:kontinuumilemma} either $K_i\cap K_j\neq\emptyset$ or we may assume that $K_i\subset\hat K_j$, but the latter contradicts maximality. Thus $K$ is a continuum. We have now proved the following proposition.

\begin{proposition}\label{prop:kontinuumipropositio}
Fix $L > C_D^{8k}$, $\delta>0$, and $B=B(x, R)\subset X$. Then there are at most $M=M(\lambda,C_D, L)<\infty$ balls $B_i$ centered at the $x$-component $U$ of $B$ such that 
\begin{enumerate}
\item $U \subset \cup_i B_i$ 
\item $\mu(B_i)\leq \frac1L\mu(B)$ for all $i$
\item For every $i$ there is a continuum $K_i \subset \overline{2^{7k}B_i}\setminus B_i$ which separates $B_i$ and $X\setminus\overline{2^{7k}B_i}$ 
\item $K_i \subset \bigcup_p\overline{5B_{x^{i}_px^{i}_{p-1}}}$, where $(x^{i}_p)_p$ is a finite $\delta$-chain 
\item $\sum_p\mu(B_{x^{i}_px^{i}_{p-1}})^{1/2}\leq C\mu(B)^{1/2}$, $C=C(\lambda,C_D,C_W)$ 
\item the union $K$ of all maximal continua in $\{K_i\}$ is a continuum. 
\end{enumerate}
\end{proposition}

Now we can finish the proof of Proposition \ref{estimate} with the following: 

\begin{lemma}
There exists a constant $C=C(\lambda, C_D, C_W)$ such that for any $\delta >0$ and $x,y \in X$, 
 \[q^{\delta}_{\mu}(x,y)\leq C\mu(B_{xy})^{1/2}.\]
\end{lemma}

\begin{proof}
Fix $x, y\in X$ and apply Proposition \ref{prop:kontinuumipropositio} to $B^1=B(x, 2^{2k}d(x, y))$ with $L=C_D^{15k}$. Note that $x$ and $y$ belong to the same component of $B^1$. Let $z=x$ or $z=y$. Let us define balls $B^{l, z}$ recursively for $l\geq 2$. Define $B^{1, z}=B^1$. Suppose we have defined the set $B^{n, z}$ for all $n\leq l$. Apply Proposition \ref{prop:kontinuumipropositio} with the same $L$ to $B^{l, z}$ to find a ball $B_{j}^{l, z}$ which contains $z$. By Lemma \ref{lemma:kontinuumilemma} $B_j^{l, z}$ is contained in the interior of some maximal continuum $K_{j'}^{l, z}$. Define $B^{l+1, z}=2^{7k}B_{j'}^{l, z}$. Note that Proposition \ref{prop:kontinuumipropositio} also yields the balls $B^{n, z}$ and $B_i^{n, z}$ and continua $K_i^{n, z}$ and $K^{n, z}$. Also, by the separation properties and Lemma \ref{lemma:kontinuumilemma}
\[
z\in B_j^{l, z}\subset \hat K_j^{l, z}\subset \hat K_{j'}^{l, z}\subset 2^{7k}B_{j'}^{l, z}=B^{l+1, z}.
\]
Let $\varepsilon>0$ and let $B_z=B(z, r_z)$ be a ball with $r_z\leq 6\delta$ and $\mu(B_z)\leq C_D^{-1}\varepsilon^2$. Define 
\[
K_{z}:=\bigcup_{n=1}^{l^\varepsilon_z}K^{n, z}
\]
where $l^\varepsilon_z$ is the smallest integer $l$ that satisfies $K^{l, z}\subset B(z, 100^{-1}r_z)$. Such a number exists, since $z\in B^{l, z}$ for all $l$. Moreover, our choice of $L$ gives $C_D^{7k}L^{-1}=\tau<1$ and 
\begin{equation}
\label{ananas}
\mu(B^{l, z})\leq C_D^{7k}L^{-1}\mu(B^{l-1, z})\leq \tau \mu(B^{l-1, z})\leq\ldots \leq \tau^{(l-1)}\mu(B^1). 
\end{equation}
In particular, $\mathrm{diam}(B^{l, z})\overset{l\rightarrow \infty}\longrightarrow 0$. 
We next show that $K_z$ is a continuum. It is clearly compact, and connectedness follows if 
\begin{equation}\label{eq:kontinuumieq}
K^{n, z}\cap K^{n+1, z}\neq\emptyset.
\end{equation}
Let $j$ be the index for which $2^{7k}B_j^{n, z}=B^{n+1, z}$. To show (\ref{eq:kontinuumieq}) it suffices to show that $K^{n, z}_j\cap K^{n+1, z}_i\neq\emptyset$ for some maximal $K_i^{n+1, z}$. By Lemma \ref{lemma:kontinuumilemma} there exists a maximal continuum $K^{n+1, z}_i$ such that the interiors of $K^{n+1, z}_i$ and $K^{n, z}_j$ intersect. Moreover either (\ref{eq:kontinuumieq}) holds or one of $K^{n+1, z}_i\subset\hat K^{n, z}_j$, $K^{n, z}_j\subset \hat K^{n+1, z}_i$ is true for any such $i$. Suppose $K^{n, z}_j\subset \hat K_i^{n+1, z}$. By separation properties $B_j^{n, z}\subset 2^{7k}B_i^{n+1, z}$, which together with our choice of $L$ leads to a contradiction:
\begin{align*}
\mu(B_j^{n, z})&\leq\mu(2^{7k}B_i^{n+1, z}) \leq C_D^{7k}\mu(B_i^{n+1, z}) \leq C_D^{7k}L^{-1}\mu(B^{n+1, z}) \\ 
&= C_D^{7k}L^{-1}\mu(2^{7k}B_j^{n, z}) \leq C^{14k}_DL^{-1}\mu(B_j^{n, z}) <\mu(B_j^{n, z}). 
\end{align*}
Now  if (\ref{eq:kontinuumieq}) were not true, $K_i^{n+1, z}\subset \hat K^{n, z}_j$ for every $i$ for which the interiors of $K^{n+1, z}_i$ and $K^{n, z}_j$ intersect. This is impossible, since \emph{every} ball $B^{n+1, z}_i$ lies in the interior of some maximal continuum and at least one of them intersects $K_j^{n, z}$. Hence (\ref{eq:kontinuumieq}) holds and $K_z$ is a continuum.

Finally, define 
\[
K=K_x\cup K_y.
\]
Note that $K$ is a continuum, since by construction $K^{1, x}=K^{1, y}$. Recall that for all $i, j, z$ there exists a finite $\delta$-chain $(x^{i, j, z}_p)_p$ in $2^{7k}B_j^{i, z}\setminus B_j^{i, z}$ such that 
\[
K^{i, z}_j\subset\bigcup_p \overline{5B_{x^{i, j, z}_px^{i, j, z}_{p-1}}}\subset \bigcup_p 6B_{x^{i, j, z}_px^{i, j, z}_{p-1}},
\]
and 
\[
\sum_p\mu(B_{x^{i, j, z}_px^{i, j, z}_{p-1}})^{1/2}\leq C\mu(B^{i, z}_j)^{1/2}.
\]
Since the set of balls
\[
\mathcal{B}:=\left\lbrace B(x^{i, j, z}_p, 6d(x^{i, j, z}_p, x^{i, j, z}_{p-1})),  B(x^{i, j, z}_{p-1}, 6d(x^{i, j, z}_p, x^{i, j, z}_{p-1})) \right\rbrace_{i, j, p, z}
\]
forms an open cover for the continuum $K$, we may extract a finite chain of balls $(A_i)_{i=1}^{N-1}$ of the set $\mathcal{B}$ so that, denoting $A_0=B_x$, $A_N=B_y$ we have $A_i\cap A_{i-1}\neq\emptyset$ for $i=1, \ldots N$. Let $x_0=x$, $x_{2N}=y$ and for other indices choose $x_{2i}\in A_i$ so that $A_i=B(x_{2i}, r_i)$ for some $r_i\leq 6\delta$. Let $x_{2i-1}\in A_i\cap A_{i-1}$ for $i=1, \ldots, N$. Now $(x_i)_{i=0}^{2N}$ is a $6\delta$-chain between the points $x$ and $y$. Moreover, by \eqref{ananas} 

\begin{align*}
&\sum_{i=1}^{2N}\mu(B_{x_ix_{i-1}})^{1/2}\leq2\sum_{i=0}^N\mu(2A_i)^{1/2} \leq C\sum_{i=1}^{N-1}\mu(A_i)^{1/2}+4\varepsilon \\
&\leq C\sum_{B\in\mathcal{B}}\mu(B)^{1/2}+4\varepsilon \leq C\sum_{z, i, j, p}\mu(B(x^{i, j, z}_p, d(x^{i, j, z}_p, x^{i, j, z}_{p-1})))^{1/2}+4\varepsilon\\
&\leq C\sum_{z, i, j}\sum_p\mu(B_{x^{i, j, z}_px^{i, j, z}_{p-1}})^{1/2}+4\varepsilon \leq C\sum_{z, i}\sum_j\mu(B_j^{i, z})^{1/2}+4\varepsilon \\
&\leq C\sum_{z}\sum_i M\mu(B^{i, z})^{1/2}+4\varepsilon \leq CM\sum_{z}\sum_i \tau^{(i-1)/2}\mu(B^1)^{1/2}+4\varepsilon \\
&\leq CM\mu(B^1)^{1/2}+4\varepsilon = CM\mu(B(x, 2^{2k}d(x, y)))^{1/2}+4\varepsilon\\
&\leq CM\mu(B_{xy})^{1/2}+4\varepsilon. 
\end{align*}
Since $\varepsilon$ is arbitrary, the claim follows.

\end{proof}



\section{Concluding remarks}
\label{sec:rem}
It is natural to ask if Theorem \ref{mainthm} remains valid with weak metric doubling measures of dimension $s \neq 2$. The two lemmas below show that it does not. 

\begin{lemma}
\label{cancun}
Let $(X,d)$ be a linearly locally connected metric space homeomorphic to $\Ss^2$, and $0<s<2$. Then $X$ does not carry weak metric doubling measures of dimension $s$. 
\end{lemma}

\begin{proof}
Assume towards a contradiction that $X$ carries such as measure $\mu$. Then there exists $C >0$ such that for every $x,y \in X$ the following holds: 
if $(x_i)_{i=0}^m$ is a $\delta$-chain from $x$ to $y$ and if $\delta$ is small enough, then  
\begin{eqnarray*}
\mu(B_{xy})^{1/2} &=& \mu(B_{xy})^{1/2-1/s}\mu(B_{xy})^{1/s} \leq C \mu(B_{xy})^{1/2-1/s} \sum_{i=1}^{m} \mu(B_{x_ix_{i-1}})^{1/s} \\
&\leq& C  \mu(B_{xy})^{1/2-1/s} \max_{i}  \mu(B_{x_ix_{i-1}})^{1/s-1/2} \sum_{i=1}^{m} \mu(B_{x_ix_{i-1}})^{1/2}. 
\end{eqnarray*}
Notice that 
$$
\max_{i}  \mu(B_{x_ix_{i-1}})^{1/s-1/2} \to 0 \quad \text{as } \delta \to 0.  
$$
Applying the estimates to all $\delta$-chains and letting $\delta \to 0$, we conclude that $\mu$ is a weak metric doubling measure of dimension $2$ and 
$$
\mu(B_{xy})^{1/2} \leq \epsilon  q_{\mu,2}(x,y) \quad \text{for all } \epsilon >0. 
$$
Since $\mu(B_{xy})>0$ for all distinct $x$ and $y$, if follows that $q_{\mu,2}(x,y)= \infty$. This contradicts Theorem 
\ref{mainthm}. 
\end{proof}

\begin{lemma}
Fix $s>2$. Then there exists a metric space $(X,d)$, homeomorphic to $\Ss^2$ and LLC, such that $X$ carries a weak metric doubling measure of dimension $s$ but there is no quasisymmetric $f:X \to \Ss^2$. 
\end{lemma}
\begin{proof}
Let $(\mathbb{R}^2,d)$ be a Rickman rug; $d$ is the product metric 
$$
d((x_1,y_1),(x_2,y_2))= \Big(|x_1-x_2|^2+|y_1-y_2|^{2/(s-1)} \Big)^{1/2}. 
$$
It is well-known that there are no quasisymmetric maps from $(\mathbb{R}^2,d)$ onto the standard plane. Moreover, it is not difficult to show that $\mu=\Ha^1 \times \Ha^{s-1}$ is a weak metric doubling measure of dimension $s$ on $(\mathbb{R}^2,d)$. To construct 
a similar example homeomorphic to $\Ss^2$, one can apply a suitable stereographic projection. 
\end{proof}

It would be interesting to extend Theorem \ref{mainthm} to higher dimensions. Recall that the Bonk-Kleiner theorem (Theorem \ref{BKtheorem}) does not extend to dimensions higher than $2$, see \cite{Se2}, \cite{HeWu}, \cite{PW}. 

\begin{question}
\label{high}
Let $(X,d)$ be a metric space homeomorphic to $\Ss^n$, $n\geq 3$. Assume that $X$ is linearly locally contractible and carries a 
weak metric doubling measure of dimension $n$. Is there a quasisymmetric $f:(X,d) \to (X,d')$, where $(X,d')$ 
is Ahlfors $n$-regular? 
\end{question}

Recall that $(X,d)$ is linearly locally contractible if there exists $\la' \ge 1$ such that $B(x,R) \subset X$ is contractible in $B(x,\la' R)$ for every $x \in X, 0<R<\diam X/\la'$. Linear local contractibility is equivalent to the LLC condition when $X$ is homeomorphic to $\Ss^2$, see \cite{BK}. 

The basic tool in the proof of Theorem \ref{mainthm} was a coarea-type estimate for real-valued functions. Extending our method to higher dimensions would require similar estimates for suitable maps with values in $\R^{n-1}$, which are difficult to construct when $n\geq 3$. This problem is related to the deep results of Semmes \cite{Se3} on Poincar\'e inequalities in Ahlfors $n$-regular and linearly locally contractible $n$-manifolds.

\vskip5pt

Finally, it is also desirable to characterize the metric spheres that can be uniformized by \emph{quasiconformal} homeomorphisms which 
are more flexible than quasisymmetric maps, see \cite{Raj}. However, it is not clear which definition of quasiconformality should be used in the generality of possibly fractal surfaces. Our methods suggest a measure-dependent modification to the familiar geometric definition. More precisely, given a measure $\mu$, conformal modulus should be defined applying not the usual path length but a \emph{$\mu$-length} as in Section \ref{sec:prel}. 

\vskip1cm 

\noindent Department of Mathematics and Statistics, University of
Jyv\"askyl\"a, P.O. Box 35 (MaD), FI-40014, University of Jyv\"askyl\"a, Finland.
\vskip0.2cm \noindent {\it E-mail address:} {\bf atte.s.lohvansuu@jyu.fi}, {\bf kai.i.rajala@jyu.fi},\\ 
{\bf martti.i.rasimus@jyu.fi}

\end{document}